\documentclass{amsart}
\usepackage{amsmath}
\usepackage{eurosym}
\usepackage{amsfonts}
\usepackage{graphicx}
\usepackage{rotating}
\usepackage{rotating}

\newtheorem{theorem}{Theorem}[section]
\newtheorem*{theorem A}{Theorem A}
\newtheorem*{theorem B}{N\"olker's Theorem}

\theoremstyle{remark}
\newtheorem{remark}{Remark}[section]
\theoremstyle{remark}

\theoremstyle{definition}

\numberwithin{equation}{section}
\def\({\left ( }
\def\){\right )}
\def\<{\left < }
\def\>{\right >}

\input{tcilatex}

\begin{document}
\title[Translation hypersurfaces with constant curvature]{Translation
hypersurfaces with constant curvature in 4-dimensional isotropic space}
\author{Muhittin Evren Aydin, Alper Osman Ogrenmis}
\address{Department of Mathematics, Faculty of Science, Firat University,
Elazig, 23200, Turkey}
\email{meaydin@firat.edu.tr, aogrenmis@firat.edu.tr}
\thanks{}

\begin{abstract}
There exist four non-equivalent types of the translation hypersurfaces in
the 4-dimensional isotropic space $\mathbb{I}^{4}$ generated by translating
the curves lying in perpendicular $k-$planes $\left(k=2,3\right) $, due to
its absolute figure. In arbitrary dimensional case; constant Gauss-Kronecker
and mean curvature translation hypersurfaces of type 1, i.e. the
hypersurfaces whose the translating curves lie in perpendicular isotropic $%
2- $planes, were investigated by the same authors in \cite{AO}. The present
study concerns such hypersurfaces in $\mathbb{I}^{4}$ of other three types.
\end{abstract}

\maketitle

\section{Introduction}

Dillen et al. \cite{DVZ} introduced a \textit{translation hypersurface} $%
M^{n-1}$ in a $n-$dimensional Euclidean space $\mathbb{R}^{n}$ as the graph
of the form%
\begin{equation}
y_{n}=f_{1}\left( y_{1}\right) +....+f_{n-1}\left( y_{n-1}\right) , 
\tag{1.1}
\end{equation}%
where $\left( y_{1},...,y_{n}\right) $ denote orthogonal coordinates in $%
\mathbb{R}^{n}$ and $f_{1},...,f_{n}$ smooth functions of single variable.
They proved that if $M^{n-1}$ is minimal, it is either a hyperplane or $%
M^{n-1}=M^{2}\times \mathbb{R}^{n-3}$, where $M^{2}$ is the \textit{Scherk's
minimal surface }(\cite{Sc}) given in explicit form%
\begin{equation*}
y_{3}=c^{-1}\left( \ln \left\vert \cos \left( cy_{2}\right) \right\vert -\ln
\left\vert \cos \left( cy_{1}\right) \right\vert \right) ,\text{ }c\in 
\mathbb{R},\text{ }c\neq 0.
\end{equation*}

In many different ambient spaces, one was tried to generalize the Scherk's
result as defining the translation (hyper)surfaces, see \cite%
{DGVDW,DVDWV,GVDW,ILM,Li,LMu,LM,Lo,Su,VWY,YF}.

In addition, Seo \cite{Se} extended the above result to the translation
hypersurfaces with arbitrary constant mean and Gauss-Kronecker curvature.

Most recently, Munteanu et al. \cite{MPRH} initated a different notion on
this framework, so-called \textit{translation graph}. Obviously, they
defined that a \textit{translation graph} in $\mathbb{R}^{p+q}$ is given in
explicit form%
\begin{equation*}
y_{p+q}\left( y_{1},y_{2},...,y_{p+q-1}\right) =f_{1}\left(
y_{1},...,y_{p}\right) +f_{2}\left( y_{p+1},...,y_{p+q-1}\right) ,
\end{equation*}%
providing certain minimality results. In addition, Moruz and Munteanu \cite%
{MM} concerned the minimal graphs of the form%
\begin{equation*}
y_{4}\left( y_{1},y_{2},y_{3}\right) =f_{1}\left( x_{1}\right) +f_{2}\left(
x_{2},x_{3}\right) ,
\end{equation*}%
which can be expressed as the sum of a curve in $y_{1}y_{4}-$plane and a
surface in $y_{2}y_{3}y_{4}-$space.

Note that the graph of the form (1.1) is formed by translating $n-1$ curves
(called \textit{generating curves}) lying in mutually perpendicular
2-planes. As the restrictions on generating curves are removed, the
different kinds of the translation hypersurfaces arise. For example; in the
particular case $n=3,$ Liu and Yu \cite{LY} introduced the notion of \textit{%
affine translation surface,} i.e., the surface whose the generating curves
lie in non-perpendicular planes. They obtained minimal affine translation
surfaces, so called \textit{affine Scherk surfaces}. Furthermore, arbitrary
constant mean curvature and Weingarten affine translation surfaces were
given in \cite{JLL,LJ}.

This notion was generalized to arbitrary dimension by the first author in 
\cite{Ay2}, defining that an\textit{\ affine translation hypersurface }in $%
\mathbb{R}^{n}$ is the graph of the form 
\begin{equation*}
y_{n}\left( y_{1},y_{2},...,y_{n-1}\right) =f_{1}\left( \zeta _{1}\right)
+...+f_{n-1}\left( \zeta _{n-1}\right) ,
\end{equation*}%
where $\zeta _{i}=\sum_{j=1}^{n-1}a_{ij}y_{j},$ $i=1,...,n-1,$ $\det
(a_{ij})\neq 0$ and $\left[ a_{ij}\right] $ is non-orthogonal matrix. He
proved that such a hypersurface with constant Gauss-Kronecker curvature in $%
\mathbb{R}^{n}$ is congruent to a cylinder.

In this study, we are interested in the counterparts of translation
hypersurfaces in isotropic geometry, i.e., a particular Cayley-Klein
geometry. For details, see \cite{K,OS,Y}. In 3-dimensional isotropic space $%
\mathbb{I}^{3},$ when the generating curves lie in perpendicular planes,
three types of translation surfaces exist up to the absolute figure:

\begin{enumerate}
\item[Type 1.] Both generating curves lie in isotropic planes; that is, a
graph of $x_{3}\left( x_{1},x_{2}\right) =f\left( x_{1}\right) +g\left(
x_{2}\right),$ where $\left( x_{1},x_{2},x_{3}\right) $ denote the
isotropically orthogonal coordinates in $\mathbb{I}^{3}$.

\item[Type 2.] One generating curve lies in non-isotropic plane and other in
isotropic plane; that is, a graph of $x_{2}\left( x_{1},x_{3}\right)
=f\left( x_{1}\right) +g\left( x_{3}\right) .$

\item[Type 3.] Both generating curves lie in non-isotropic planes; that is,
a graph of $x_{1}\left( x_{2},x_{3}\right) =\frac{1}{2}\left( f\left(
x_{2}+x_{3}-\pi /2\right) +g\left( \pi /2-x_{2}+x_{3}\right) \right) .$
\end{enumerate}

As well as the non-isotropic planes, Strubecker \cite{St} obtained the
minimal translation surfaces in $\mathbb{I}^{3},$ so called \textit{%
isotropic Scherk's surfaces of type 1,2,3} and respectively given in
explicit form $x_{3}=c\left( x_{1}^{2}-x_{2}^{2}\right)$, $%
x_{2}=c^{-1}\left( \ln \left\vert cx_{3}\right\vert -\ln \left\vert \cos
cx_{1}\right\vert \right) $ and%
\begin{equation*}
x_{1}=\left( 2c\right) ^{-1}\left( \ln \left\vert \cos c\left(
x_{2}+x_{3}-\pi /2\right) \right\vert -\ln \left\vert \cos c\left( \pi
/2-x_{2}+x_{3}\right) \right\vert \right) ,\text{ }c\in \mathbb{R},\text{ }%
c\neq 0.
\end{equation*}

Afterwards, his results were generalized by Sipus \cite{MS} to the
translation surfaces in $\mathbb{I}^{3}$ with arbitrary constant Gaussian
and mean curvature. The situation that the generating curves in $\mathbb{I}%
^{3}$ are non-planar or lie in non-perpendicular planes extends the above
categorization and the results, see \cite{Ay1}.

In $\mathbb{I}^{4},$ there are four types of translation hypersurfaces whose
the generating curves lie in mutually perpendicular $k-$planes $\left(
k=2,3\right) ,$ see Section 3. In more generale case, i.e. in arbitrary
dimensional isotropic spaces, the translation hypersurfaces whose the
generating curves lie in isotropic $2-$planes, said to be \textit{of type 1,}
were investigated in \cite{AO}. The present study deals with other three
types of translation hypersurfaces in $\mathbb{I}^{4}$ with constant
Gauss-Kronecker and mean curvature.

In addition, due to the absolute figure of $\mathbb{I}^{n}$ $n\geq 3,$ the
graph hypersurfaces associated with the form $x_{n}=f\left(
x_{1},...,x_{n-1}\right) $ differ from other hypersurfaces, for a smooth
real-valued function $f$. For example; the Gauss-Kronecker and mean
curvature for such a graph hypersurface in $\mathbb{I}^{n}$ correspond to
the determinant and trace of the Hessian of $f$, respectively. The formulas
of these curvatures were initiated by Chen et al. \cite{CDV}, besides
obtaining flat and minimal graphs associated with most famous production
models in microecenonomics.

As far as we know, this is first study formulating such fundamental
curvatures of a generic hypersurface in $\mathbb{I}^{n}$.

\section{Preliminaries}

Some differential geometric approaches on the curves and the hypersurfaces
in isotropic geometry can be found in \cite{D,EDH,MSD1,MSD2,PGM,PO,S1,S2}.

Let $\mathbb{P}^{n}$ denote the $n-$dimensional real projective space, $%
\omega $ a hyperplane in $\mathbb{P}^{n}$ and $\mathbb{I}^{n}=\mathbb{P}%
^{n}\backslash \omega $ the obtained affine space. We call $\mathbb{I}^{n}$ 
\textit{\ $n-$dimensional isotropic space} if $\omega $ contains a
hypersphere $\mathbb{S} $ with null radius. Then the pair $\left\{ \omega ,%
\mathbb{S}\right\} $ is called \textit{absolute figure} of $\mathbb{I}^{n}$
determined in the homogeneous coordinates by%
\begin{equation*}
\omega :u_{0}=0,\text{ }\mathbb{S}:u_{0}=u_{1}^{2}+...+u_{n-1}^{2}=0.
\end{equation*}%
The vertex of $\mathbb{S}$ is $F\left( 0:0:...:1\right) $ which we call 
\textit{absolute point}.

Denote the affine coordinates $x_{1}=\frac{u_{1}}{u_{0}},...,x_{n}=\frac{%
u_{n}}{u_{0}},$ $u_{0}\neq 0.$ Then the group of motions of $\mathbb{I}^{n}$
which preserves the absolute figure is given in terms of affine coordinates
by%
\begin{equation*}
\begin{bmatrix}
A & 0 \\ 
B & 1%
\end{bmatrix}%
,
\end{equation*}%
where $A$ is an orthonogal $\left( n-1,n-1\right) -$matrix, $B$ a real $%
\left( 1,n-1\right) -$matrix.

Let $p=\left( p_{1},...,p_{n}\right) ,$ $q=\left( q_{1},...,q_{n}\right) $
be two points in $\mathbb{I}^{n}.$ The \textit{isotropic distance} between $%
p $ and $q$ is defined by%
\begin{equation*}
d_{i}\left( p,q\right) =\sqrt{\sum_{i=1}^{n-1}\left( p_{i}-q_{i}\right) ^{2}}%
.
\end{equation*}%
If $d_{i}=0,$ then the so-called \textit{range} between $p$ and $q$ is
defined as $d_{i}^{r}=\left\vert p_{n}-q_{n}\right\vert .$

A line is said to be \textit{isotropic} if its point at infinity is
absolute. Other lines are \textit{non-isotropic}. We call a $k-$plane 
\textit{isotropic }(\textit{non-isotropic}) if it contains (does not) an
isotropic line. In the affine model of $\mathbb{I}^{n}$, the isotropic lines
and the isotropic $k-$planes are parallel to the $x_{n}-$axis. For example;
the following 
\begin{equation*}
a_{1}x_{1}+...+a_{n}x_{n}=b, \text{ } a_{i},b \in \mathbb{R},
\end{equation*}
determines an isotropic (non-isotropic) hyperplane if $a_{n}=0$ $\left( \neq
0\right) .$

Note that the hyperplane $x_{n}=0,$ so-called basic hyperplane, is Euclidean
and therefore the Euclidean metric is used in it.

As distinct from the Euclidean case, the orthogonality in $\mathbb{I}^{n}$
does not mean the perpendicularity. Obviously, two non-isotropic lines are
orthogonal if their projections onto the basic hyperplane are perpendicular
up to the Euclidean metric. Nevertheless, an isotropic line is orthogonal to
some non-isotropic line. As a consequence, each non-isotropic hyperplane is
orthogonal to the isotropic one. In addition, two isotropic hyperplanes are
orthogonal if their projections onto the basic hyperplane are perpendicular.

We call a curve \textit{isotropic} (\textit{non-isotropic}) $k-$\textit{%
planar} if it lies in an isotropic (non-isotropic) $k-$plane.

\subsection{Curvature theory of hypersurfaces}

This part of isotropic geometry is close to the Euclidean case.

Let $M^{n-1}$, $n \ge 3$, be a hypersurface in $\mathbb{I}^{n}$ whose the
tangent hyperplane at each point is non-isotropic. Then the coefficients $%
g_{ij}$ of the first fundamental form are calculated by the induced metric
from $\mathbb{I}^{n}.$ The normal vector field $U$ is completely isotropic,
i.e. $\left( 0,0,...,1\right) .$

For the second fundamental form, let us consider a curve $r$ on $M^{n-1}$
with isotropic arclength $s$ and the tangent vector $t\left( s\right)
=r^{\prime }\left( s\right) =\frac{dr}{ds}.$ Denote $S$ the projection of $%
r^{\prime \prime }\left( s\right) =\frac{d^{2}r}{ds^{2}}$ onto the tangent
hyperplane of $M^{n-1}.$ Then, the following decomposition occurs:%
\begin{equation*}
r^{\prime \prime }\left( s\right) =\kappa _{g}\sigma +\kappa _{n}U,
\end{equation*}%
where $\kappa _{g}$ and $\kappa _{n}$ are \textit{geodesic} and \textit{%
normal curvatures} of $r,$ respectively. Hence, it follows $\kappa
_{g}=\left\Vert r^{\prime \prime }\left( s\right) \right\Vert _{i},$ where
we mean the induced norm by $\left\Vert \cdot \right\Vert _{i}.$ In
addition, by a direct computation, we have%
\begin{equation}
\kappa _{n}=\frac{1}{\sqrt{\det g_{ij}}}\sum_{i,j=1}^{n-1}\det \left(
r_{x_{1}},...,r_{x_{n-1}},r_{x_{i}x_{j}}\right) \frac{dx_{i}}{ds}\frac{dx_{j}%
}{ds},  \tag{2.1}
\end{equation}%
where $r_{x_{i}}=\frac{\partial r}{\partial x_{i}}$ and $r_{x_{i}x_{j}}=%
\frac{\partial ^{2}r}{\partial x_{i}\partial x_{j}},$ $1\leq i,j\leq n-1.$
If we put 
\begin{equation*}
h_{ij}=\frac{\det \left( r_{x_{1}},...,r_{x_{n-1}},r_{x_{i}x_{j}}\right) }{%
\sqrt{\det g_{ij}}}
\end{equation*}%
into (2.1) then one can be written in the matrix form as%
\begin{equation}
\kappa _{n}=\tilde{t}^{T}\cdot \left[ h_{ij}\right] \cdot \tilde{t},\text{ }%
\tilde{t}=\left( \frac{dx_{1}}{ds},...,\frac{dx_{n-1}}{ds}\right) ^{T}, 
\tag{2.2}
\end{equation}%
where "$\cdot $" denotes the matrix multiplication. If $r$ is a curve with
arbitrary parameter, then (2.2) turns to 
\begin{equation*}
\kappa _{n}=\frac{\tilde{t}^{T}\cdot \left[ h_{ij}\right] \cdot \tilde{t}}{%
\tilde{t}^{T}\cdot \left[ g_{ij}\right] \cdot \tilde{t}}.
\end{equation*}

The extreme values of $\kappa _{n}$ which we call \textit{principal
curvatures} correspond to the eigenvalues of the matrix $\left[ h_{ij}\right]
\cdot \left[ g_{ij}\right] ^{-1}.$ Denote the principal curvatures $\kappa
_{1},...,\kappa _{n-1}$ and $\left[ a_{ij}\right] =\left[ h_{ij}\right]
\cdot \left[ g_{ij}\right] ^{-1}$. Therefore, the characteristic equation of 
$\left[ a_{ij}\right] $ follows%
\begin{equation*}
\det \left( \left[ a_{ij}\right] -\lambda I_{n-1}\right) =\lambda ^{n-1}-tr%
\left[ a_{ij}\right] \lambda ^{n-2}+...+\left( -1\right) ^{n-1}\det \left[
a_{ij}\right] =0,
\end{equation*}%
which provides the fundamental curvatures, called \textit{isotropic
Gauss-Kronecker curvature} (or \textit{relative curvature}) and \textit{%
isotropic mean curvature}. We shortly call them \textit{Gauss-Kronecker}\ $%
\left( K\right) $ and \textit{mean curvature }$\left( H\right) $\textit{.}
Obviously, one obtains%
\begin{equation*}
K=\kappa _{1}...\kappa _{n-1}=\det \left[ a_{ij}\right] \text{ or }K=\frac{%
\det \left[ h_{ij}\right] }{\det \left[ g_{ij}\right] }
\end{equation*}%
and%
\begin{equation*}
\left( n-1\right) H=\kappa _{1}+...+\kappa _{n-1}=tr\left[ a_{ij}\right] ,
\end{equation*}%
where $tr$ denotes the trace of a matrix.

A hypersurface is said to be flat (minimal) if $K$ $\left( H\right)$ is
identically zero.

Notice that the isotropic counterpart for the notion of \textit{shape
operator} in the Euclidean (or Riemannian) sense of a hypersurface is indeed
a zero map. In $\mathbb{I}^{n}$, the matrix $\left[ a_{ij}\right]$ however
plays its role.

\section{Categorization of translation hypersurfaces}

Let $M^{3}$ be a translation hypersurface in $\mathbb{I}^{4}$ generated by
translating three curves lying in perpendicular $k-$planes, $k=2,3.$ Denote
the generating curves $\alpha ,\beta ,\gamma .$ Up to the absolute figure of 
$\mathbb{I}^{4}$, there are four types of such hypersurfaces listed as below:

\begin{enumerate}
\item[Type 1.] Three of $\alpha ,\beta ,\gamma $ are isotropic 2-planar.
Then $M^{3}$ is parameterized by%
\begin{equation*}
r\left( u,v,w\right) =\left( u,v,w,f\left( u\right) +g\left( u\right)
+h\left( w\right) \right) ,
\end{equation*}%
where $\alpha ,\beta $ and $\gamma $ lie in $x_{1}x_{4}-$plane, $x_{2}x_{4}-$%
plane and $x_{3}x_{4}-$plane, respectively.

\item[Type 2.] $\alpha $ is non-isotropic 2-planar and $\beta ,\gamma $
isotropic 2-planar. Then $M^{3}$ is parameterized by%
\begin{equation*}
r\left( u,v,w\right) =\left( u+v,w,f\left( u\right) ,g\left( u\right)
+h\left( w\right) \right) ,
\end{equation*}%
where $\alpha ,\beta $ and $\gamma $ lie in $x_{1}x_{3}-$plane, $x_{1}x_{4}-$%
plane and $x_{2}x_{4}-$plane, respectively. The regularity implies that $f$
is a non-constant function.

\item[Type 3.] $\alpha ,\beta $ are non-isotropic 2-planar and $\gamma $
isotropic 2-planar. Then $M^{3}$ is parameterized by%
\begin{equation*}
r\left( u,v,w\right) =\left( u+v+w,f\left( u\right) ,g\left( v\right)
,h\left( w\right) \right) ,
\end{equation*}%
where $\alpha ,\beta $ and $\gamma $ lie in $x_{1}x_{2}-$plane, $x_{1}x_{3}-$%
plane and $x_{1}x_{4}-$plane, respectively. The regularity implies that
neither $f$ nor $g$ is a constant function.

\item[Type 4.] Three of $\alpha ,\beta ,\gamma $ are non-isotropic
hyperplanar. The curves $\alpha ,\beta ,\gamma $ and the hyperplanes $%
P_{\alpha },P_{\beta },P_{\gamma }$ containing them can be choosen as%
\begin{equation*}
\left. 
\begin{array}{l}
\alpha \left( u\right) =\left( f\left( u\right) ,u,u,u+\pi \right) ,\text{ }%
P_{\alpha }:-2x_{2}+x_{3}+x_{4}=\pi ; \\ 
\beta \left( v\right) =\left( g\left( v\right) ,v,v,-v+\frac{\pi }{3}\right)
,\text{ }P_{\beta }:2x_{2}+x_{3}+3x_{4}=\pi ; \\ 
\gamma \left( w\right) =\left( h\left( w\right) ,6w,-w,w-\frac{\pi }{2}%
\right) ,\text{ }P_{\gamma }:x_{2}+4x_{3}-2x_{4}=\pi .%
\end{array}%
\right.
\end{equation*}%
Then $M^{3}$ is parameterized by%
\begin{equation*}
\left. 
\begin{array}{l}
r\left( u,v,w\right) = \\ 
\left( f\left( u\right) +g\left( v\right) +h\left( w\right)
,u+v+6w,u+v-w,u-v+w+\frac{5\pi }{6}\right) .%
\end{array}%
\right.
\end{equation*}%
The regularity implies that $\frac{df}{du} -\frac{dg}{dv} \neq 0$.
\end{enumerate}

A translation hypersurface of above one type is no equivalent to that of
other type due to the absolute figure of $\mathbb{I}^4$.

We hereinafter denote the derivatives of $f,g,h$ with respect to the given
variable by a prime and so.

\section{Translation hypersurfaces of type 2}

The Gauss-Kronecker and the mean curvature respectively follows%
\begin{equation}
K=\frac{g^{\prime }f^{\prime \prime }g^{\prime \prime }h^{\prime \prime }}{%
\left( f^{\prime }\right) ^{3}}  \tag{4.1}
\end{equation}%
and%
\begin{equation}
3H=\frac{f^{\prime \prime }g^{\prime }}{\left( f^{\prime }\right) ^{3}}%
+g^{\prime \prime }\frac{1+\left( f^{\prime }\right) ^{2}}{\left( f^{\prime
}\right) ^{2}}+h^{\prime \prime }.  \tag{4.2}
\end{equation}

\begin{theorem}
A flat translation hypersurface of type 2 in $\mathbb{I}^{4}$ is a
cylindrical hypersurface with non-isotropic rulings. Furthermore; if one has
nonzero constant Gauss-Kronecker curvature, then the following occurs:%
\begin{equation*}
f\left( u\right) =\lambda u^{\frac{1}{2}},\text{ }g\left( v\right) =\mu v^{%
\frac{3}{2}},\text{ }h\left( w\right) =\xi w^{2},
\end{equation*}%
where $\lambda ,\mu ,\xi \in \mathbb{R}$ and $\lambda \mu \xi \neq 0.$
\end{theorem}

\begin{proof}
(4.1) follows that $K$ vanishes if at least one of $f,g,h$ is a linear
function with respect to the given variable, that is, at least one of the
generating curves turns to a non-isotropic line. Now, let assume that the
Gauss-Kronecker curvature is a nonzero constant $K_{0}.$ So, (4.1) leads to%
\begin{equation}
\frac{f^{\prime \prime }}{\left( f^{\prime }\right) ^{3}}=\lambda ,\text{ }%
g^{\prime }g^{\prime \prime }=\mu ,\text{ }h^{\prime \prime }=\xi , 
\tag{4.3}
\end{equation}%
for $\lambda ,\mu ,\xi \in \mathbb{R}$ and $K_{0}=\lambda \mu \xi \neq 0.$
After solving (4.3), we obtain%
\begin{equation*}
f\left( u\right) =\pm \frac{1}{\lambda }\sqrt{-2\lambda u+c_{1}}+c_{2},\text{
}g\left( v\right) =\pm \frac{1}{3\mu }\left( 2\mu v+c_{3}\right) ^{\frac{3}{2%
}}+c_{4}
\end{equation*}%
and%
\begin{equation*}
h\left( w\right) =\frac{\xi }{2}w^{2}+c_{5}w+c_{6},
\end{equation*}%
where $c_{1},...,c_{6}\in \mathbb{R}$. This concludes the proof.
\end{proof}

\begin{theorem}
Let $M^{3}$ be a minimal translation hypersurface of type 2 in $\mathbb{I}%
^{4}.$ Then it is either a non-isotropic hyperplane or $M^{3}=S^{2}\times 
\mathbb{R}$, where $S^{2}$ is the isotropic Scherk's surface of type 2 in $%
\mathbb{I}^{3}$.
\end{theorem}

\begin{proof}
(4.2) leads to%
\begin{equation}
\frac{f^{\prime \prime }g^{\prime }}{\left( f^{\prime }\right) ^{3}}%
+g^{\prime \prime }\frac{1+\left( f^{\prime }\right) ^{2}}{\left( f^{\prime
}\right) ^{2}}+h^{\prime \prime }=0,  \tag{4.4}
\end{equation}%
which implies $h^{\prime \prime }=h_{0},$ $h_{0}\in \mathbb{R}$. To solve
(4.4) we distinguish two cases depending on $h_{0}=0$ or not.

\begin{enumerate}
\item $h_{0}=0.$ (4.4) can be rewritten as%
\begin{equation}
\frac{f^{\prime \prime }}{f^{\prime }\left[ 1+\left( f^{\prime }\right) ^{2}%
\right] }+\frac{g^{\prime \prime }}{g^{\prime }}=0.  \tag{4.5}
\end{equation}%
The situation that $f^{\prime \prime }=g^{\prime \prime }=0$ is a solution
to (4.5), which leads $M^{3}$ to be a non-isotropic hyperplane. If $%
f^{\prime \prime }g^{\prime \prime }\neq 0,$ (4.5) implies%
\begin{equation}
\frac{f^{\prime \prime }}{f^{\prime }\left[ 1+\left( f^{\prime }\right) ^{2}%
\right] }=\lambda =-\frac{g^{\prime \prime }}{g^{\prime }}  \tag{4.6}
\end{equation}%
for $\lambda \in \mathbb{R},$ $\lambda \neq 0.$ By solving (4.6), we derive%
\begin{equation}
f\left( u\right) =\pm \frac{1}{\lambda }\arccos \left( c_{1}e^{\lambda
u}\right) ,\text{ }g\left( v\right) =-\frac{c_{2}}{\lambda }e^{-\lambda v}, 
\tag{4.7}
\end{equation}%
for $c_{1},c_{2}\in \mathbb{R},$ $c_{1}c_{2}\neq 0.$ Up to suitable
translations and constants, $M^{3}$ can be written by a change of parameter
in (4.7) as 
\begin{equation*}
r\left( \tilde{u},\tilde{v},w\right) =\left( \frac{1}{\lambda }\ln
\left\vert \frac{\cos \lambda \tilde{u}}{\lambda \tilde{v}}\right\vert ,0,%
\tilde{u},\tilde{v}\right) +w\left( 0,1,0,0\right) .
\end{equation*}%
This implies the hypothesis of the theorem.

\item $h_{0}\neq 0.$ (4.4) can be rearranged as%
\begin{equation}
A\left( u\right) +B\left( v\right) =-h_{0}C\left( v\right) D\left( u\right) ,
\tag{4.8}
\end{equation}%
where 
\begin{equation*}
A\left( u\right) =\frac{f^{\prime \prime }}{f^{\prime }\left[ 1+\left(
f^{\prime }\right) ^{2}\right] },\text{ }B\left( v\right) =\frac{g^{\prime
\prime }}{g^{\prime }},\text{ }C\left( v\right) =\frac{1}{g^{\prime }},\text{
}D\left( u\right) =\frac{\left( f^{\prime }\right) ^{2}}{1+\left( f^{\prime
}\right) ^{2}}.
\end{equation*}%
One deduces from (4.9) that $A,B,C,D$ are all constant. The fact that $C,D$
are constant yields $f^{\prime \prime }=0$ and $g^{\prime \prime }=0.$ This
however contradicts with $h_{0}\neq 0$ in (4.8).
\end{enumerate}
\end{proof}

\begin{theorem}
Let $M^{3}$ be a translation hypersurface of type 2 in $\mathbb{I}^{4}$ with
nonzero constant mean curvature $H_{0}$. Then, for $\lambda ,\mu ,\xi \in 
\mathbb{R},$ one of the following occurs:

\begin{enumerate}
\item[(i)] $f=f\left( u\right) ,$ $f^{\prime }\neq 0,$ $g\left( v\right)
=\lambda ,$ $h\left( w\right) =\frac{3H_{0}}{2}w^{2};$

\item[(ii)] $f\left( u\right) =\lambda u,$ $g\left( v\right) =\mu v,$ $%
h\left( w\right) =\frac{3H_{0}}{2}w^{2},$ $\lambda \mu \neq 0;$

\item[(iii)] $f\left( u\right) =\lambda u^{\frac{1}{2}},$ $g\left( v\right)
=\mu v,$ $h\left( w\right) =\xi w^{2},$ $\lambda \mu \neq 0,$ $\xi \neq 
\frac{3H_{0}}{2};$

\item[(iv)] $f\left( u\right) =\lambda u,$ $g\left( v\right) =\mu v^{2},$ $%
h\left( w\right) =\xi w^{2},$ $\lambda \mu \neq 0,$ $\xi \neq \frac{3H_{0}}{2%
};$

\item[(v)] $M^{3}=S^{2}$ $\times P,$ where $S^{2}$ is the isotropic Scherk'
s surface of type 2 in $\mathbb{I}^{3}$ and $P$ is a parabolic circle in $%
\mathbb{I}^{2}$ with isotropic curvature $3H_{0}.$
\end{enumerate}
\end{theorem}

\begin{proof}
Reconsidering (4.2) leads to $h^{\prime \prime }=h_{0},$ $h_{0}\in \mathbb{R}
$ and therefore we get \ 
\begin{equation}
3H_{0}=\frac{f^{\prime \prime }g^{\prime }}{\left( f^{\prime }\right) ^{3}}%
+g^{\prime \prime }\frac{1+\left( f^{\prime }\right) ^{2}}{\left( f^{\prime
}\right) ^{2}}+h_{0}.  \tag{4.9}
\end{equation}%
To solve (4.9), we have two cases:

\begin{enumerate}
\item $g^{\prime }=g_{0},$ $g_{0}\in \mathbb{R}$. In particular; if $%
g_{0}=0, $ then we conclude $h_{0}=3H_{0}$ and 
\begin{equation*}
h\left( w\right) =\frac{3}{2}H_{0}w^{2}+c_{1}w+c_{1},\text{ }c_{1},c_{2}\in 
\mathbb{R},
\end{equation*}%
which implies the statement (i) of the theorem. Nevertheless; if $g_{0}\neq
0 $ then, by (4.9), we get%
\begin{equation}
\frac{3H_{0}-h_{0}}{g_{0}}=\frac{f^{\prime \prime }}{\left( f^{\prime
}\right) ^{3}}.  \tag{4.10}
\end{equation}%
If $3H_{0}=h_{0}$ in (4.10), we immediately achieve the proof of the
statement (ii) of the theorem. Otherwise, after solving (4.10), we obtain%
\begin{equation*}
f\left( u\right) =\pm \frac{g_{0}}{3H_{0}-h_{0}}\sqrt{\frac{-6H_{0}+2h_{0}}{%
g_{0}}u+c_{3}}+c_{4},
\end{equation*}%
where $3H_{0}\neq h_{0}$ and $c_{3},c_{4}\in \mathbb{R}.$ This gives the
proof of the statement (iii) of the theorem.

\item $g^{\prime \prime }\neq 0.$ We consider two cases:

\begin{enumerate}
\item $f^{\prime }=f_{0}\neq 0,$ $f_{0}\in \mathbb{R}$. (4.9) leads to%
\begin{equation*}
3H_{0}=\frac{1+f_{0}^{2}}{f_{0}^{2}}g^{\prime \prime }+h_{0},
\end{equation*}%
which implies the proof of the statement (iv) of the theorem.

\item $f^{\prime \prime }\neq 0.$ (4.9) implies $h_{0}=3H_{0}$ and%
\begin{equation}
\frac{f^{\prime \prime }}{\left( f^{\prime }\right) ^{3}}=\lambda \frac{%
1+\left( f^{\prime }\right) ^{2}}{\left( f^{\prime }\right) ^{2}}\text{ and }%
g^{\prime \prime }=-\lambda g^{\prime },  \tag{4.11}
\end{equation}%
where $\lambda \in \mathbb{R}$, $\lambda \neq 0.$ After solving (4.11), we
obtain 
\begin{equation}
f\left( u\right) =\pm \frac{1}{\lambda }\arccos \left( c_{1}e^{\lambda
u}\right) ,\text{ }g\left( v\right) =-\frac{c_{2}}{\lambda }e^{-\lambda v} 
\tag{4.12}
\end{equation}%
for $c_{1},c_{2}\in \mathbb{R},$ $c_{1}c_{2}\neq 0.$ Up to suitable
translations and constants, $M^{3}$ can be written by a change of parameter
in (4.12) as%
\begin{equation*}
r\left( \tilde{u},\tilde{v},w\right) =\left( \frac{1}{\lambda }\ln
\left\vert \frac{\cos \lambda \tilde{u}}{\lambda \tilde{v}}\right\vert ,0,%
\tilde{u},\tilde{v}\right) +\left( 0,w,0,\frac{3}{2}H_{0}w^{2}\right) ,
\end{equation*}%
which completes the proof of the theorem.
\end{enumerate}
\end{enumerate}
\end{proof}

\section{Translation hypersurfaces of type 3}

The Gauss-Kronecker and the mean curvature are respectively%
\begin{equation}
K=\frac{\left( h^{\prime }\right) ^{2}f^{\prime \prime }g^{\prime \prime
}h^{\prime \prime }}{\left( f^{\prime }g^{\prime }\right) ^{3}}  \tag{5.1}
\end{equation}%
and%
\begin{equation}
3H=h^{\prime }\left[ \frac{f^{\prime \prime }}{\left( f^{\prime }\right) ^{3}%
}+\frac{g^{\prime \prime }}{\left( g^{\prime }\right) ^{3}}\right]
+h^{\prime \prime }\left[ 1+\frac{1}{\left( f^{\prime }\right) ^{2}}+\frac{1%
}{\left( g^{\prime }\right) ^{2}}\right] .  \tag{5.2}
\end{equation}%
Note that the roles of $f$ and $g$ are symmetric in (5.2) and henceforth we
only discuss the situations depending on $f$ while solving it.

\begin{theorem}
A flat translation hypersurface of type 3 in $\mathbb{I}^{4}$ is a
cylindrical hypersurface with non-isotropic rulings. Furthermore; if one has
nonzero constant Gauss-Kronecker curvature, then the following occurs:%
\begin{equation*}
f\left( u\right) =\lambda u^{\frac{1}{2}},\text{ }g\left( v\right) =\mu v^{%
\frac{1}{2}},\text{ }h\left( w\right) =\xi w^{\frac{4}{3}},
\end{equation*}%
where $\lambda ,\mu ,\eta \in \mathbb{R}$ and $\lambda \mu \xi \neq 0.$
\end{theorem}

\begin{proof}
(5.1) follows that $K$ vanishes if at least one of $f,g,h$ is a linear
function with respect to the given variable; that is, at least one of the
generating curves turns to be a non-isotropic line. Now, let us assume that
it is a nonzero constant. So, (5.1) leads to%
\begin{equation}
\frac{f^{\prime \prime }}{\left( f^{\prime }\right) ^{3}}=\lambda ,\text{ }%
\frac{g^{\prime \prime }}{\left( g^{\prime }\right) ^{3}}=\mu ,\text{ }%
\left( h^{\prime }\right) ^{2}h^{\prime \prime }=\xi ,  \tag{5.3}
\end{equation}%
for $\lambda ,\mu ,\xi \in \mathbb{R}$ and $\lambda \mu \xi \neq 0.$ After
solving (5.3), we obtain%
\begin{equation*}
f\left( u\right) =\pm \frac{1}{\lambda }\sqrt{-2\lambda u+c_{1}}+c_{2},\text{
}g\left( v\right) =\pm \frac{1}{\lambda }\sqrt{-2\mu v+c_{3}}+c_{4}
\end{equation*}%
and%
\begin{equation*}
h\left( w\right) =\frac{1}{4\xi }\left( 3\xi w+c_{5}\right) ^{\frac{4}{3}%
}+c_{6},
\end{equation*}%
where $c_{1},...,c_{6}\in \mathbb{R}$. This concludes the proof.
\end{proof}

\begin{theorem}
Let $M^{3}$ be a minimal translation hypersurface of type 3 in $\mathbb{I}%
^{4}.$ Then, it is either a non-isotropic hyperplane or one of the following
occurs:

\begin{enumerate}
\item[(i)] $f=f\left( u\right) ,$ $f^{\prime }\neq 0,$ $g=g\left( v\right) ,$
$g^{\prime }\neq 0,$ $h\left( w\right) =\lambda ;$

\item[(ii)] $M^{3}=S^{2}\times \mathbb{R}$, where $S^{2}$ is the isotropic
Scherk' s surface of type 2 in $\mathbb{I}^{3}$;

\item[(iii)] $f\left( u\right) =\lambda \left( -u\right) ^{\frac{1}{2}},$ $%
g\left( v\right) =\lambda v^{\frac{1}{2}},$ $h\left( w\right) =\mu w,$ $%
\lambda \mu \neq 0;$

\item[(iv)] $f\left( u\right) =\eta \ln \left\vert \frac{1+\sqrt{1+\kappa
e^{\lambda u}}}{1-\sqrt{1+\kappa e^{\lambda u}}}\right\vert $ or $f\left(
u\right) =\kappa e^{\lambda u},$ $g\left( v\right) =\mu \ln \left\vert \frac{%
1+\sqrt{1+\xi e^{\varpi v}}}{1-\sqrt{1+\xi e^{\varpi v}}}\right\vert $ or $%
g\left( v\right) =\xi e^{\varpi v},$ $h\left( w\right) =\rho e^{\tau w},$
where $\eta ,\kappa ,\lambda ,\mu ,\xi ,\varpi ,\rho ,\tau $ are nonzero
constants.
\end{enumerate}
\end{theorem}

\begin{proof}
Due to $H=0,$ (5.2) implies%
\begin{equation}
h^{\prime }\left[ \frac{f^{\prime \prime }}{\left( f^{\prime }\right) ^{3}}+%
\frac{g^{\prime \prime }}{\left( g^{\prime }\right) ^{3}}\right] +h^{\prime
\prime }\left[ 1+\frac{1}{\left( f^{\prime }\right) ^{2}}+\frac{1}{\left(
g^{\prime }\right) ^{2}}\right] =0.  \tag{5.4}
\end{equation}%
The situation that $h^{\prime }=0$ is a solution to (5.4), which is the
proof of the statement (i) of the theorem. Assume that $h^{\prime }\neq 0$.
In order to solve (5.4), we have to consider the following cases:

\begin{enumerate}
\item $f^{\prime }=f_{0},$ $f_{0}\in \mathbb{R}$, $f_{0}\neq 0.$ Then, (5.4)
reduces to%
\begin{equation}
\frac{g^{\prime \prime }h^{\prime }}{\left( g^{\prime }\right) ^{3}}%
+h^{\prime \prime }\left[ 1+\frac{1}{f_{0}^{2}}+\frac{1}{\left( g^{\prime
}\right) ^{2}}\right] =0.  \tag{5.5}
\end{equation}%
To solve (5.5), we have two possibilities: The first one is the situation
that $M^{3}$ is a non-isotropic hyperplane. The second one is that $%
g^{\prime \prime }h^{\prime \prime }\neq 0.$ So, (5.5) can be rewritten as%
\begin{equation}
\frac{f_{0}^{2}g^{\prime \prime }}{g^{\prime }\left[ \left(
1+f_{0}^{2}\right) \left( g^{\prime }\right) ^{2}+f_{0}^{2}\right] }=\lambda
=-\frac{h^{\prime \prime }}{h^{\prime }},  \tag{5.6}
\end{equation}%
where $\lambda \in \mathbb{R}$, $\lambda \neq 0.$ After solving (5.6), we
obtain%
\begin{equation*}
g\left( v\right) =\pm \frac{f_{0}}{\sqrt{1+f_{0}^{2}}}\arccos \left( c_{1}%
\left[ 1+f_{0}^{2}\right] e^{\lambda v}\right) ,\text{ }h\left( w\right) =%
\frac{-c_{2}}{\lambda }e^{-\lambda w},
\end{equation*}%
where $c_{1},c_{2}\in \mathbb{R}$, $c_{1}c_{2}\neq 0,$ which is the proof of
the statement (ii).

\item $f^{\prime \prime }\neq 0.$ By symmetry, we have $g^{\prime \prime
}\neq 0$ and distinguish two cases:

\begin{enumerate}
\item $h^{\prime }=h_{0},$ $h_{0}\in \mathbb{R}$, $h_{0}\neq 0.$ (5.4)
implies 
\begin{equation}
\frac{f^{\prime \prime }}{\left( f^{\prime }\right) ^{3}}=\lambda =-\frac{%
g^{\prime \prime }}{\left( g^{\prime }\right) ^{3}}.  \tag{5.7}
\end{equation}%
Solving (5.7) leads to 
\begin{equation*}
f\left( u\right) =\pm \frac{1}{\lambda }\sqrt{-2\lambda u+c_{1}}+c_{2},\text{
}g\left( v\right) =\pm \frac{1}{\lambda }\sqrt{2\lambda v+c_{3}}+c_{4},
\end{equation*}%
where $c_{1},...,c_{4}\in \mathbb{R}$, which indicates the proof of the
statement (iii) of the theorem.

\item $h^{\prime \prime }\neq 0.$ (5.4) yields that $\frac{h^{\prime \prime }%
}{h^{\prime }}=\mu ,$ $\mu \in \mathbb{R}$, $\mu \neq 0,$ or $h\left(
w\right) =c_{1}e^{\mu w},$ $c_{1}\in \mathbb{R},$ $c_{1}\neq 0.$ Thereby,
(5.4) reduces to%
\begin{equation}
\frac{f^{\prime \prime }}{\left( f^{\prime }\right) ^{3}}+\frac{\mu }{\left(
f^{\prime }\right) ^{2}}+\frac{g^{\prime \prime }}{\left( g^{\prime }\right)
^{3}}+\frac{\mu }{\left( g^{\prime }\right) ^{2}}=-\mu ,  \tag{5.8}
\end{equation}%
which implies%
\begin{equation}
\frac{f^{\prime \prime }}{\left( f^{\prime }\right) ^{3}}+\frac{\mu }{\left(
f^{\prime }\right) ^{2}}=\xi ,  \tag{5.9}
\end{equation}%
and%
\begin{equation}
\frac{g^{\prime \prime }}{\left( g^{\prime }\right) ^{3}}+\frac{\mu }{\left(
g^{\prime }\right) ^{2}}=\rho ,  \tag{5.10}
\end{equation}%
where $\xi \in \mathbb{R}$ and $\rho =-\mu -\xi .$ From (5.9), we have 
\begin{equation}
f^{\prime }\left( u\right) =\pm \left( \frac{\xi }{\mu }+\frac{c_{2}}{\mu }%
e^{2\mu u}\right) ^{\frac{-1}{2}},\text{ }c_{2}\in \mathbb{R},\text{ }%
c_{2}\neq 0.  \tag{5.11}
\end{equation}%
If $\xi =0$ in (5.11), then we can derive $f\left( u\right) =\mp \left( 
\frac{c_{2}}{\mu }\right) ^{\frac{-1}{2}}e^{-\mu u}.$ Otherwise, we get%
\begin{equation*}
f\left( u\right) =-\frac{1}{\sqrt{\mu \xi }}\tanh ^{-1}\left( \sqrt{1+\frac{%
c_{2}}{\xi }e^{2\mu u}}\right) =-\frac{1}{2\sqrt{\mu \xi }}\ln \left\vert 
\frac{1+\sqrt{1+\frac{c_{2}}{\xi }e^{2\mu u}}}{1-\sqrt{1+\frac{c_{2}}{\xi }%
e^{2\mu u}}}\right\vert .
\end{equation*}%
Same solutions are also satisfied to (5.10) and therefore we complete the
proof.
\end{enumerate}
\end{enumerate}
\end{proof}

\begin{theorem}
Let a translation hypersurface of type 3 have nonzero constant mean
curvature $H_{0}$ in $\mathbb{I}^{4}$. Then, for $\lambda ,\mu ,\xi \in 
\mathbb{R},$ one of the following occurs:

\begin{enumerate}
\item[(i)] $f\left( u\right) =\lambda u,$ $g\left( v\right) =\mu v,$ $%
h\left( w\right) =\frac{3H_{0}\left( \lambda \mu \right) ^{2}}{\left(
\lambda \mu \right) ^{2}+\lambda ^{2}+\mu ^{2}}w^{2},$ $\lambda \mu \neq 0;$

\item[(ii)] $f\left( u\right) =\lambda u,$ $g\left( v\right) =\left( \frac{%
-2\mu }{3H_{0}}v\right) ^{\frac{1}{2}},$ $h\left( w\right) =\mu w,$ $\lambda
\mu \neq 0;$

\item[(iii)] $f\left( u\right) =\lambda u^{\frac{1}{2}},$ $g\left( v\right)
=\mu v^{\frac{1}{2}},$ $h\left( w\right) =\xi w,$ $\lambda \mu \xi \neq 0.$
\end{enumerate}
\end{theorem}

\begin{proof}
Due to $H_{0}\neq 0,$ $h$ cannot be constant in (5.2). Nevertheless, we have
to distinguish several cases to solve (5.2):

\begin{enumerate}
\item $f^{\prime }=f_{0},$ $f_{0}\in \mathbb{R}$, $f_{0}\neq 0.$ Then (5.2)
follows%
\begin{equation}
3H_{0}=\frac{g^{\prime \prime }h^{\prime }}{\left( g^{\prime }\right) ^{3}}%
+h^{\prime \prime }\left[ \lambda +\frac{1}{\left( g^{\prime }\right) ^{2}}%
\right] .  \tag{5.12}
\end{equation}%
where $\lambda =\frac{f_{0}^{2}+1}{f_{0}^{2}}.$ In order to solve (5.12) the
regularity provides two cases:

\begin{enumerate}
\item $g^{\prime }=g_{0},$ $g_{0}\in \mathbb{R}$, $g_{0}\neq 0.$ (5.12)
yields%
\begin{equation*}
h\left( w\right) =\frac{3H_{0}}{\mu }w^{2}+c_{1}w+c_{2},
\end{equation*}%
where $c_{1},c_{2},\mu \in \mathbb{R},$ $\mu =\lambda +\frac{1}{g_{0}^{2}}.$
This is the proof of statement (i) of the theorem.

\item $g^{\prime \prime }\neq 0.$ We have two cases:

\begin{enumerate}
\item $h^{\prime }=h_{0},$ $h_{0}\in \mathbb{R}$, $h_{0}\neq 0.$ By (5.12),
we derive%
\begin{equation}
\frac{3H_{0}}{h_{0}}=\frac{g^{\prime \prime }}{\left( g^{\prime }\right) ^{3}%
}.  \tag{5.13}
\end{equation}%
Solving (5.13) leads to%
\begin{equation*}
g\left( v\right) =\pm \frac{h_{0}}{3H_{0}}\sqrt{\frac{-6H_{0}}{h_{0}}v+c_{3}}%
,
\end{equation*}%
where $c_{3}\in \mathbb{R}$ and this proves the statement (ii) of the
theorem.

\item $h^{\prime \prime }\neq 0.$ Dividing (5.12) with $h^{\prime }$ and
taking partial derivative respect to $w$ gives the following polynomial
equation on $\left( g^{\prime }\right) ^{2}$%
\begin{equation*}
\left( \frac{3H_{0}h^{\prime \prime }}{\left( h^{\prime }\right) ^{2}}%
+\lambda \left( \frac{h^{\prime \prime }}{h^{\prime }}\right) ^{\prime
}\right) \left( g^{\prime }\right) ^{2}+\left( \frac{h^{\prime \prime }}{%
h^{\prime }}\right) ^{\prime }=0,
\end{equation*}%
which yields a contradiction.
\end{enumerate}
\end{enumerate}

\item $f^{\prime \prime }\neq 0.$ The symmetry gives $f^{\prime \prime
}g^{\prime \prime }\neq 0.$ We have two cases:

\begin{enumerate}
\item $h^{\prime }=h_{0},$ $h_{0}\in \mathbb{R},$ $h_{0}\neq 0.$ (5.2)
reduces to%
\begin{equation}
\frac{3H_{0}}{h_{0}}=\frac{f^{\prime \prime }}{\left( f^{\prime }\right) ^{3}%
}+\frac{g^{\prime \prime }}{\left( g^{\prime }\right) ^{3}}  \tag{5.14}
\end{equation}%
Solving (5.14) gives%
\begin{equation*}
f\left( u\right) =\pm \frac{1}{\frac{3H_{0}}{h_{0}}-\lambda }\sqrt{-2\left( 
\frac{3H_{0}}{h_{0}}-\lambda \right) u+c_{1}}+c_{2}
\end{equation*}%
and%
\begin{equation*}
g\left( v\right) =\pm \frac{1}{\lambda }\sqrt{-2\lambda v+c_{3}}+c_{4},
\end{equation*}%
for $\lambda ,c_{1},...,c_{4}\in \mathbb{R},$ $\lambda \neq 0$. This is the
proof of the statement (iii) of the theorem.

\item $h^{\prime \prime }\neq 0.$ Dividing (5.2) with $h^{\prime }$ and
taking its partial derivative with respect to $w,$ we deduce%
\begin{equation}
-3H_{0}\frac{h^{\prime \prime }}{\left( h^{\prime }\right) ^{2}}=\left( 
\frac{h^{\prime \prime }}{h^{\prime }}\right) ^{\prime }\left[ 1+\frac{1}{%
\left( f^{\prime }\right) ^{2}}+\frac{1}{\left( g^{\prime }\right) ^{2}}%
\right] .  \tag{5.15}
\end{equation}%
Both-hand side must be nonzero in (5.15) and thus we can rewrite it as
follows:%
\begin{equation}
-3H_{0}\frac{h^{\prime \prime }}{\left( h^{\prime }\right) ^{2}}\left[
\left( \frac{h^{\prime \prime }}{h^{\prime }}\right) ^{\prime }\right]
^{-1}=1+\frac{1}{\left( f^{\prime }\right) ^{2}}+\frac{1}{\left( g^{\prime
}\right) ^{2}}.  \tag{5.16}
\end{equation}%
This is a contradiction due to the fact that the right-hand side of (5.16)
cannot be a constant.
\end{enumerate}
\end{enumerate}
\end{proof}

\section{Translation hypersurfaces of type 4}

The Gauss-Kronecker and the mean curvature are respectively%
\begin{equation}
K=\frac{8f^{\prime \prime }g^{\prime \prime }h^{\prime \prime }}{49\left(
f^{\prime }-g^{\prime }\right) ^{5}}  \tag{6.1}
\end{equation}%
and%
\begin{equation}
\left. 
\begin{array}{c}
3H=\frac{2}{49\left( f^{\prime }-g^{\prime }\right) ^{3}}\left\{ \left[
37\left( g^{\prime }\right) ^{2}+2\left( h^{\prime }\right) ^{2}-10g^{\prime
}h^{\prime }+49\right] f^{\prime \prime }+\right. \\ 
\left. +\left[ 37\left( f^{\prime }\right) ^{2}+2\left( h^{\prime }\right)
^{2}-10f^{\prime }h^{\prime }+49\right] g^{\prime \prime }+2h^{\prime \prime
}\left( f^{\prime }-g^{\prime }\right) ^{2}\right\} .%
\end{array}%
\right.  \tag{6.2}
\end{equation}%
As in previous section, the roles of $f$ and $g$ are symmetric in (6.2) and,
while solving it, the situations depending on $f$ are only considered.

\begin{theorem}
There does not exist a translation hypersurface of type 4 in $\mathbb{I}^{4}$
with constant Gauss-Kronecker curvature, except the cylindrical
hypersurfaces with non-isotropic rulings.
\end{theorem}

\begin{proof}
Assume that $K=K_{0}\neq 0$ and thus $f^{\prime \prime }g^{\prime \prime
}h^{\prime \prime }\neq 0$. $\left( 6.1\right) $ follows%
\begin{equation}
\frac{49K_{0}}{8h_{0}}=\frac{f^{\prime \prime }g^{\prime \prime }}{\left(
f^{\prime }-g^{\prime }\right) ^{5}},  \tag{6.3}
\end{equation}%
where $h^{\prime \prime }=h_{0}\neq 0.$ The partial derivative of (6.3) with
respect to $u$ yields 
\begin{equation}
f^{\prime \prime \prime }\left( f^{\prime }-g^{\prime }\right) -5\left(
f^{\prime \prime }\right) ^{2}=0.  \tag{6.4}
\end{equation}%
The fact that the coefficient of the term $g^{\prime }$ in (6.4) must be
zero leads to the contradiction $f^{\prime \prime }=0.$
\end{proof}

\begin{theorem}
Let a translation hypersurface of type 4 be minimal in $\mathbb{I}^{4}.$
Then it is either a non-isotropic hyperplane or, for $\lambda ,\mu ,\xi \in 
\mathbb{R},$ one of the following occurs:

\begin{enumerate}
\item[(i)] $f\left( u\right) =\lambda u,$ $g\left( v\right) =\lambda v-\frac{%
1}{\mu }\ln \left\vert \mu v\right\vert ,$ $h\left( w\right) =\frac{5\lambda 
}{2}w+\frac{1}{\mu }\ln \left\vert \cos \xi w\right\vert ,$ $\mu \xi \neq 0;$

\item[(ii)] $f\left( u\right) =\lambda u-\frac{1}{\mu }\ln \left\vert \cos
\xi w\right\vert ,$ $g\left( v\right) =\lambda v+\frac{1}{\mu }\ln
\left\vert \cos \xi w\right\vert ,$ $h\left( w\right) =\frac{37\lambda }{5}%
w, $ $\mu \xi \neq 0.$
\end{enumerate}
\end{theorem}

\begin{remark}
If $\lambda =0$ in the statement (ii) of Theorem 6.2, then $%
M^{3}=S^{2}\times \mathbb{R}$, where $S^{2}$ is \textit{isotropic Scherk' s
surface of type 3} in $\mathbb{I}^{4}$ with codimension 2. For details, see
Appendix 1.
\end{remark}

\begin{proof}
(6.2) follows%
\begin{equation}
\left. 
\begin{array}{l}
0=\left[ 37\left( g^{\prime }\right) ^{2}+2\left( h^{\prime }\right)
^{2}-10g^{\prime }h^{\prime }+49\right] f^{\prime \prime } \\ 
+\left[ 37\left( f^{\prime }\right) ^{2}+2\left( h^{\prime }\right)
^{2}-10f^{\prime }h^{\prime }+49\right] g^{\prime \prime }+2h^{\prime \prime
}\left( f^{\prime }-g^{\prime }\right) ^{2}.%
\end{array}%
\right.  \tag{6.5}
\end{equation}%
We have two cases to solve (6.5):

\begin{enumerate}
\item $f^{\prime }=f_{0},$ $f_{0}\in \mathbb{R}.$ (6.5) reduces to%
\begin{equation}
\frac{g^{\prime \prime }}{\left( f_{0}-g^{\prime }\right) ^{2}}+\frac{%
2h^{\prime \prime }}{2\left( h^{\prime }\right) ^{2}-10f_{0}h^{\prime
}+37f_{0}^{2}+49}=0.  \tag{6.6}
\end{equation}%
The situation that $g^{\prime \prime }=h^{\prime \prime }=0,$ $g^{\prime
}\neq f_{0},$ leads $M^{3}$ to be a non-isotropic hyperplane. If $g^{\prime
\prime }h^{\prime \prime }\neq 0,$ (6.6) implies%
\begin{equation}
\frac{g^{\prime \prime }}{\left( f_{0}-g^{\prime }\right) ^{2}}=\lambda =%
\frac{-2h^{\prime \prime }}{2\left( h^{\prime }\right) ^{2}-10f_{0}h^{\prime
}+37f_{0}^{2}+49},  \tag{6.7}
\end{equation}%
where $\lambda \in \mathbb{R},$ $\lambda \neq 0.$ After solving (6.7), we
conclude%
\begin{equation*}
g\left( v\right) =f_{0}v-\frac{1}{\lambda }\ln \left\vert \lambda
v+c_{1}\right\vert +c_{2}
\end{equation*}%
and%
\begin{equation*}
h\left( w\right) =\frac{5f_{0}}{2}w+\frac{1}{\lambda }\ln \left\vert \cos
\left( -\frac{7\lambda \sqrt{2+f_{0}^{2}}}{2}w+c_{3}\right) \right\vert
+c_{4},
\end{equation*}%
where $c_{1},...,c_{4}\in \mathbb{R}$. This is the proof of the statement
(i) of the theorem.

\item $f^{\prime \prime }\neq 0.$ By symmetry, we deduce $g^{\prime \prime
}\neq 0.$ We have two cases:

\begin{enumerate}
\item $h^{\prime }=h_{0},$ $h_{0}\in \mathbb{R}.$ (6.5) can be rewritten as%
\begin{equation}
\frac{f^{\prime \prime }}{37\left( f^{\prime }\right) ^{2}-10h_{0}f^{\prime
}+49+h_{0}^{2}}=\lambda =\frac{-g^{\prime \prime }}{37\left( g^{\prime
}\right) ^{2}-10h_{0}f^{\prime }+49+h_{0}^{2}}.  \tag{6.8}
\end{equation}%
Solving (6.8), we conclude%
\begin{equation*}
f\left( u\right) =\frac{-1}{37\lambda }\ln \left\vert \cos \left( \mu
\lambda u+c_{1}\right) \right\vert +\frac{5h_{0}}{37}u+c_{2}
\end{equation*}%
and%
\begin{equation*}
g\left( v\right) =\frac{1}{37\lambda }\ln \left\vert \cos \left( -\mu
\lambda v+c_{3}\right) \right\vert +\frac{5h_{0}}{37}v+c_{4},
\end{equation*}%
where $\mu =\sqrt{1813+12h_{0}^{2}},$ which proves the statement (ii) of the
theorem.

\item $h^{\prime \prime }\neq 0.$ This case yields a contradiction, see
Appendix 2.
\end{enumerate}
\end{enumerate}
\end{proof}

\begin{theorem}
Let a translation hypersurface of type 4 in $\mathbb{I}^{4}$ have nonzero
constant mean curvature. Then, for $\lambda ,\mu ,\xi \in \mathbb{R},$ one
of the following occurs:
\end{theorem}

\begin{enumerate}
\item[(i)] $f\left( u\right) =\lambda u,$ $g\left( v\right) =\mu v,$ $%
h\left( w\right) =\xi w^{2},$ $\lambda \neq \mu ,$ $\xi \neq 0;$

\item[(ii)] $f\left( u\right) =\lambda u,$ $g\left( v\right) =\lambda v+\mu
v^{\frac{1}{2}},$ $h\left( w\right) =\xi w,$ $\mu \neq 0$.
\end{enumerate}

\begin{proof}
We have several cases to solve (6.2):

\begin{enumerate}
\item $f^{\prime }=f_{0}\in \mathbb{R}.$ (6.2) reduces to%
\begin{equation}
\lambda \left( f_{0}-g^{\prime }\right) ^{3}=\left[ 2\left( h^{\prime
}\right) ^{2}-10f_{0}h^{\prime }+37f_{0}^{2}+49\right] g^{\prime \prime
}+2h^{\prime \prime }\left( f_{0}-g^{\prime }\right) ^{2},  \tag{6.9}
\end{equation}%
where $\lambda =\frac{147H_{0}}{2}\neq 0.$ If $g^{\prime }=g_{0}\in \mathbb{R%
},$ $f_{0}\neq g_{0}$ in (6.9), then we immediately have the proof of the
statement (i) of the theorem. Next we assume $g^{\prime \prime }\neq 0$ and
consider the following cases:

\begin{enumerate}
\item $h^{\prime }=h_{0}\in \mathbb{R}.$ (6.9) follows%
\begin{equation}
\frac{g^{\prime \prime }}{\left( f_{0}-g^{\prime }\right) ^{3}}=\mu , 
\tag{6.10}
\end{equation}%
for $\mu =\frac{\lambda }{37f_{0}^{2}+2h_{0}^{2}-10h_{0}f_{0}+49}.$ Solving
(6.10) leads to%
\begin{equation*}
g\left( v\right) =f_{0}v\pm \frac{1}{\mu }\left( 2\mu v+c_{1}\right) ^{\frac{%
1}{2}}+c_{2},\text{ }c_{1},c_{2}\in \mathbb{R},
\end{equation*}%
which is the proof of the statement (ii) of the theorem.

\item $h^{\prime \prime }\neq 0.$ The partial derivative of (6.9) with
respect to $w$\ gives%
\begin{equation}
\frac{g^{\prime \prime }}{\left( f_{0}-g^{\prime }\right) ^{2}}+\frac{%
h^{\prime \prime \prime }}{h^{\prime \prime }\left( 2h^{\prime
}-5f_{0}\right) }=0.  \tag{6.11}
\end{equation}%
where $h^{\prime \prime \prime }\neq 0$ due to $g^{\prime \prime }\neq 0.$
(6.11) implies%
\begin{equation}
\frac{g^{\prime \prime }}{\left( f_{0}-g^{\prime }\right) ^{2}}=\rho =-\frac{%
h^{\prime \prime \prime }}{h^{\prime \prime }\left( 2h^{\prime
}-5f_{0}\right) },  \tag{6.12}
\end{equation}%
where $\rho \in \mathbb{R},$ $\rho \neq 0.$ Considering (6.12) into (6.9)
leads to%
\begin{equation}
\lambda \left( f_{0}-g^{\prime }\right) =\left[ 2\left( h^{\prime }\right)
^{2}-10h^{\prime }f_{0}+37f_{0}^{2}+49\right] \rho +2h^{\prime \prime }, 
\tag{6.13}
\end{equation}%
which is no possible since the left-hand side of (6.13) cannot be a constant.
\end{enumerate}

\item $f^{\prime \prime }\neq 0.$ The symmetry follows $g^{\prime \prime
}\neq 0.$ We have two cases depending on $h^{\prime \prime }=0$ or not.
These cases however imply some contradictions, see Appendix 3 and Appendix 4.
\end{enumerate}
\end{proof}

\section{Appendix}

This appendix provides a detailed explanation for some calculations ignored
in the proofs of Theorem 6.2 and Theorem 6.3. \bigskip

\noindent \textbf{Appendix 1.} \textit{Isotropic Scherk's surface of type 3}
in $\mathbb{I}^{4}$ \textit{with codimension} $2$. \newline
\noindent The formulas of the Gaussian and the mean curvatures for a surface
in $\mathbb{I}^{4}$ with codimension 2 can be found in \cite{AM}.

The absolute of $\mathbb{I}^{4}$ gives rise to three types of the
translation surfaces whose both generating curves lie in perpendicular
hyperplanes:

\begin{enumerate}
\item[Type 1.] Both generating curves lie in isotropic hyperplanes
determined by the equations:%
\begin{equation*}
x_{1}+x_{2}+x_{3}=\pi ,\text{ }2x_{1}-x_{2}-x_{3}=\pi .
\end{equation*}%
The obtained translation surface in $\mathbb{I}^{4}$ is parameterized by%
\begin{equation*}
r\left( u,v\right) =\left( u+v,u+v,-2u+v,f\left( u\right) +g\left( v\right)
\right) ,
\end{equation*}%
where the generating curves are%
\begin{equation*}
\alpha \left( u\right) =\left( u,u,-2u+\pi ,f\left( u\right) \right) ,\text{ 
}\beta \left( v\right) =\left( v,v,v-\pi ,g\left( v\right) \right) .
\end{equation*}

\item[Type 2.] One generating curve lies in non-isotropic hyperplane and
other one in isotropic hyperplane determined by the equations:%
\begin{equation*}
x_{2}+x_{3}+x_{4}=\pi ,\text{ }x_{1}+x_{2}-x_{3}=\pi .
\end{equation*}%
The obtained translation surface in $\mathbb{I}^{4}$ is parameterized by%
\begin{equation*}
r\left( u,v\right) =\left( f\left( u\right) +v+\pi ,u+v,u+2v,g\left(
v\right) -2u+\pi \right) ,
\end{equation*}%
where the generating curves are%
\begin{equation*}
\alpha \left( u\right) =\left( f\left( u\right) ,u,u,-2u+\pi \right) ,\text{ 
}\beta \left( v\right) =\left( v+\pi ,v,2v,g\left( v\right) \right) .
\end{equation*}

\item[Type 3.] Both generating curves lie in non-isotropic hyperplanes
determined by the equations:%
\begin{equation*}
-2x_{2}+x_{3}+x_{4}=\pi ,\text{ }2x_{2}+x_{3}+3x_{4}=\pi .
\end{equation*}%
The obtained translation surface in $\mathbb{I}^{4}$ is parameterized by%
\begin{equation*}
r\left( u,v\right) =\left( f\left( u\right) +g\left( v\right) ,u+v,u+v,u-v+%
\frac{4\pi }{3}\right) ,
\end{equation*}%
where the generating curves are%
\begin{equation*}
\alpha \left( u\right) =\left( f\left( u\right) ,u,u,u+\pi \right) ,\text{ }%
\beta \left( v\right) =\left( g\left( v\right) ,v,v,-v+\frac{\pi }{3}\right)
.
\end{equation*}
\end{enumerate}

If a translation surface of type 3 in $\mathbb{I}^{4}$ is minimal then we
can achieve 
\begin{equation}
\frac{f^{\prime \prime }}{\left( f^{\prime }\right) ^{2}+2}+\frac{g^{\prime
\prime }}{\left( g^{\prime }\right) ^{2}+2}=0.  \tag{7.1}
\end{equation}%
The solution of (7.1) parameterizes so-called \textit{the isotropic Scherk's
surface of type 3} in $\mathbb{I}^{4}.$

\begin{remark}
The study of above three types of translation surfaces in $\mathbb{I}^{4}$
with prescribed curvature could be an interesting problem.
\end{remark}

\noindent \textbf{Appendix 2.} $h^{\prime \prime }\neq 0.$ \newline
\noindent The partial derivative of (6.5) with respect to $w$ gives%
\begin{equation}
2h^{\prime }h^{\prime \prime }\left( f^{\prime \prime }+g^{\prime \prime
}\right) -5h^{\prime \prime }\left( f^{\prime \prime }g^{\prime }+f^{\prime
}g^{\prime \prime }\right) +h^{\prime \prime \prime }\left( f^{\prime
}-g^{\prime }\right) ^{2}=0.  \tag{7.2}
\end{equation}%
If $h^{\prime \prime \prime }=0$ in (7.2), then it reduces to%
\begin{equation*}
2h^{\prime }\left( f^{\prime \prime }+g^{\prime \prime }\right) -5\left(
f^{\prime \prime }g^{\prime }+f^{\prime }g^{\prime \prime }\right) =0,
\end{equation*}%
which implies $f^{\prime \prime }+g^{\prime \prime }=0$ and $f^{\prime
\prime }g^{\prime }+f^{\prime }g^{\prime \prime }=0$ or $f^{\prime
}-g^{\prime }=0.$ That is no possible due to the regularity. Hence, we
deduce $h^{\prime \prime \prime }\neq 0.$ Dividing (7.2) with $h^{\prime
\prime }$ and then taking its parital derivative with respect to $w,$ we
conclude%
\begin{equation}
2h^{\prime \prime }\left( f^{\prime \prime }+g^{\prime \prime }\right)
+\left( \frac{h^{\prime \prime \prime }}{h^{\prime \prime }}\right) ^{\prime
}\left( f^{\prime }-g^{\prime }\right) ^{2}=0,  \tag{7.3}
\end{equation}%
which yields two cases:

\begin{enumerate}
\item $f^{\prime \prime }=-g^{\prime \prime }.$ It follows from (7.3) that $%
h^{\prime \prime \prime }=\lambda h^{\prime \prime },$ $\lambda \neq 0.$
Putting those into (7.2) leads to%
\begin{equation*}
5f^{\prime \prime }+\lambda \left( f^{\prime }-g^{\prime }\right) =0,
\end{equation*}%
which is no possible.

\item $f^{\prime \prime }\neq -g^{\prime \prime }.$ (7.3) can rewritten as%
\begin{equation*}
\frac{f^{\prime \prime }+g^{\prime \prime }}{\left( f^{\prime }-g^{\prime
}\right) ^{2}}+\frac{\left( h^{\prime \prime \prime }/h^{\prime \prime
}\right) ^{\prime }}{2h^{\prime \prime }}=0,
\end{equation*}%
which leads to%
\begin{equation}
f^{\prime \prime }+g^{\prime \prime }=\mu \left( f^{\prime }-g^{\prime
}\right) ^{2},\text{ }\mu \neq 0.  \tag{7.4}
\end{equation}%
The partial derivative of (7.4) with respect to $u$ yields%
\begin{equation*}
f^{\prime \prime \prime }=2\mu f^{\prime }f^{\prime \prime }-2\mu f^{\prime
\prime }g^{\prime },
\end{equation*}%
in which the fact that the coefficient of the term $g^{\prime }$ must vanish
leads to the contradiction $f^{\prime \prime }=0.$
\end{enumerate}

\bigskip

\noindent \textbf{Appendix 3.} $f^{\prime \prime }g^{\prime \prime }\neq 0$
and $h^{\prime }=h_{0}\in \mathbb{R}.$ \newline
\noindent (6.2) can be rearranged as%
\begin{equation}
\frac{\lambda \left( f^{\prime }-g^{\prime }\right) ^{3}}{f^{\prime \prime
}g^{\prime \prime }}=\frac{A\left( u\right) }{f^{\prime \prime }}+\frac{%
B\left( v\right) }{g^{\prime \prime }},  \tag{7.5}
\end{equation}%
where $A\left( u\right) =37\left( f^{\prime }\right) ^{2}-10h_{0}f^{\prime
}+2h_{0}^{2}+49$ and $B\left( v\right) =37\left( g^{\prime }\right)
^{2}-10h_{0}g^{\prime }+2h_{0}^{2}+49$. We have two cases:

\begin{enumerate}
\item $f^{\prime \prime }=f_{0}\in \mathbb{R},$ $f_{0}\neq 0.$ The partial
derivative of (7.5) with respect to $u$ and $v$ gives 
\begin{equation*}
f^{\prime }g^{\prime \prime \prime }-g^{\prime }g^{\prime \prime \prime
}+2\left( g^{\prime \prime }\right) ^{2}=0,
\end{equation*}%
which yields the contradiction $g^{\prime \prime }=0$ due to the fact that
the coefficient of the term $f^{\prime }$ must vanish.

\item $f^{\prime \prime \prime }g^{\prime \prime \prime }\neq 0.$ The
partial derivative of (7.5) with respect to $u$ and $v$ gives%
\begin{equation}
6=3\left( C\left( u\right) -D\left( v\right) \right) \left( f^{\prime
}-g^{\prime }\right) +C\left( u\right) D\left( v\right) \left( f^{\prime
}-g^{\prime }\right) ^{2},  \tag{7.6}
\end{equation}%
where $C\left( u\right) =\frac{f^{\prime \prime \prime }}{\left( f^{\prime
\prime }\right) ^{2}}$ and $D\left( v\right) =\frac{g^{\prime \prime \prime }%
}{\left( g^{\prime \prime }\right) ^{2}}.$ If $C=C_{0}\in \mathbb{R},$ $%
C_{0}\neq 0,$ then, by taking twice partial derivative of (7.5) with respect
to $u$ leads to the contradiction $g^{\prime \prime \prime }=0.$ Therefore, $%
C$ is no constant, neither is $D$ by symmetry. Taking partial derivative of
(7.6) with respect to $u$ and $v$ and then dividing with the product $%
C^{\prime }D^{\prime }$ gives%
\begin{equation*}
2\underset{E\left( u\right) }{\underbrace{\frac{\left( f^{\prime }C\right)
^{\prime }}{C^{\prime }}}}.\underset{F\left( v\right) }{\underbrace{\frac{%
\left( g^{\prime }D\right) ^{\prime }}{D^{\prime }}}}=\underset{G\left(
u\right) }{\underbrace{\frac{\left[ \left( f^{\prime }\right) ^{2}C\right]
^{\prime }-3f^{\prime \prime }}{C^{\prime }}}}+\underset{I\left( v\right) }{%
\underbrace{\frac{\left[ \left( g^{\prime }\right) ^{2}D\right] ^{\prime
}-3g^{\prime \prime }}{D^{\prime }}}},
\end{equation*}%
where $E,F,G,I$ must be constant, i.e. $E=E_{0},$ $F=F_{0},$ $G=G_{0},$ $%
I=I_{0}.$ This yields 
\begin{equation}
C\left( u\right) =\frac{3}{f^{\prime }-E_{0}},\text{ }D\left( v\right) =%
\frac{3}{g^{\prime }-F_{0}}.  \tag{7.7}
\end{equation}%
Substituting (7.7) into (7.6) yields the following%
\begin{equation*}
\frac{2E_{0}F_{0}}{3}+\left( \frac{E_{0}}{3}-F_{0}\right) g^{\prime }+\left[
-E_{0}+\frac{F_{0}}{3}\right] f^{\prime }+\frac{2f^{\prime }g^{\prime }}{3}%
=0,
\end{equation*}%
which leads to the contradiction $f^{\prime \prime }=0$ or $g^{\prime \prime
}=0.$
\end{enumerate}

\bigskip

\noindent \textbf{Appendix 4.} $f^{\prime \prime }g^{\prime \prime
}h^{\prime \prime }\neq 0.$ \newline
\noindent The partial derivative of (6.2) with respect to $w$\ gives%
\begin{equation}
\left[ 2h^{\prime }-5g^{\prime }\right] f^{\prime \prime }+\left[ 2h^{\prime
}-5f^{\prime }\right] g^{\prime \prime }+\frac{h^{\prime \prime \prime }}{%
h^{\prime \prime }}\left( f^{\prime }-g^{\prime }\right) ^{2}=0.  \tag{7.8}
\end{equation}%
If $h^{\prime \prime }=h_{0}\in \mathbb{R},$ $h_{0}\neq 0,$ (7.8) reduces to%
\begin{equation}
\left[ 2h^{\prime }-5g^{\prime }\right] f^{\prime \prime }+\left[ 2h^{\prime
}-5f^{\prime }\right] g^{\prime \prime }=0.  \tag{7.9}
\end{equation}%
The partial derivative of (7.9) with respect to $w$ leads to $f^{\prime
\prime }=-g^{\prime \prime }$ and therefore the contradiction $f^{\prime
}-g^{\prime }=0$ is obtained. Henceforth, we asssume $h^{\prime \prime
\prime }\neq 0.$ There are two more cases:

\begin{enumerate}
\item $h^{\prime \prime \prime }=\lambda h^{\prime \prime },$ $\lambda \in 
\mathbb{R},$ $\lambda \neq 0.$ Then (7.8) reduces to%
\begin{equation}
\left[ 2h^{\prime }-5g^{\prime }\right] f^{\prime \prime }+\left[ 2h^{\prime
}-5f^{\prime }\right] g^{\prime \prime }+\lambda \left( f^{\prime
}-g^{\prime }\right) ^{2}=0.  \tag{7.10}
\end{equation}%
The partial derivative of (7.10) with respect to $w$ gives $f^{\prime \prime
}=\mu =-g^{\prime \prime },$ $\mu \in \mathbb{R},$ $\mu \neq 0.$ Thus,
(7.10) yields%
\begin{equation*}
5g^{\prime \prime }-\lambda \left( f^{\prime }-g^{\prime }\right) =0,
\end{equation*}%
which implies the contradiction $\lambda =0.$

\item $\left( \frac{h^{\prime \prime \prime }}{h^{\prime \prime }}\right)
^{\prime }\neq 0.$ The partial derivative of (7.8) with respect to $w$ gives%
\begin{equation}
2h^{\prime \prime }\left( f^{\prime \prime }+g^{\prime \prime }\right)
+\left( \frac{h^{\prime \prime \prime }}{h^{\prime \prime }}\right) ^{\prime
}\left( f^{\prime }-g^{\prime }\right) ^{2}=0,  \tag{7.11}
\end{equation}%
in which $\left( \frac{h^{\prime \prime \prime }}{h^{\prime \prime }}\right)
^{\prime }=\xi h^{\prime \prime },$ $\xi \in \mathbb{R},$ $\xi \neq 0.$
Thereby (7.11) reduces to%
\begin{equation}
2\left( f^{\prime \prime }+g^{\prime \prime }\right) +\xi \left( f^{\prime
}-g^{\prime }\right) ^{2}=0.  \tag{7.12}
\end{equation}%
The partial derivative of (7.12) with respect to $u$ and $v$ concludes the
contradiction $\xi =0.$
\end{enumerate}

\end{document}